\documentclass{article}
\usepackage[T2A]{fontenc}
\usepackage[cp1251]{inputenc}
\usepackage{amssymb}
\usepackage{amsxtra}
\usepackage{amsmath}
\usepackage{amsthm}
\usepackage{hyperref}

\usepackage[russian,english]{babel}

\DeclareMathOperator{\UU}{U}

\DeclareMathOperator{\WW}{W}

\DeclareMathOperator{\Spec}{Spec}

\DeclareMathOperator{\im}{im\,}
\DeclareMathOperator{\Aut}{Aut\,}

\DeclareMathOperator{\Gal}{Gal}
\DeclareMathOperator{\lev}{lev}
\DeclareMathOperator{\der}{der}

\newcommand{\ad}{ad}
\renewcommand{\sc}{sc}

\DeclareMathOperator{\ZZ}{{\mathbb Z}}

\DeclareMathOperator{\FF}{{\mathbb F}}

\newtheorem{lem}{Lemma}
\newtheorem*{thm*}{Theorem}
\newtheorem{thm}{Theorem}

\newtheorem*{cor*}{Corollary}

\let\l\left
\let\r\right

\let\eps\varepsilon
\let\st\scriptstyle

\title{Elementary subgroup of an isotropic reductive group is perfect}
\author{A. Luzgarev\thanks{The author is supported by RFBR 09-01-00784, RFBR 09-01-00878 and RFBR 09-01-90304.},
A. Stavrova\thanks{The author is supported by RFBR 09-01-00878 and RFBR 09-01-90304.}}
\date{December 31, 2009}
\begin{document}
\selectlanguage{english}
\maketitle

\section{Introduction}

Let $R$ be a commutative ring with 1, and let $G$ be an isotropic reductive algebraic group over $R$.
In~\cite{PS} Victor Petrov and the second author introduced a notion of an elementary subgroup $E(R)$
of the group of points $G(R)$.
In this note we prove that, as one might expect from the split case (e.g.,~\cite{Stein}) as
well as from the field case (e.g.,~\cite{Tits64}), under natural assumptions
the elementary subgroup of a reductive group is perfect.

More precisely, assume that $G$ is isotropic in the following strong sense:
it possesses a parabolic subgroup that intersects properly any semisimple normal subgroup of $G$.
Such a parabolic subgroup $P$ is called {\it strictly proper}. Denote
by $E_P(R)$ the subgroup of $G(R)$ generated by the $R$-points of the unipotent
radicals of $P$ and of an opposite parabolic subgroup $P^-$.
The main theorem of~\cite{PS} states that $E_P(R)$
does not depend on the choice of $P$, as soon as
for any maximal ideal $M$ of $R$ all irreducible components of the relative
root system of $G_{R_M}$ (see~\cite[Exp. XXVI, \S 7]{SGA} for the definition) are of rank $\ge 2$.
Under this assumption, we call $E_P(R)$ the {\it elementary
subgroup} of $G(R)$ and denote it simply by $E(R)$.
In particular, $E(R)$ is normal in $G(R)$.
This definition of $E(R)$ generalizes the well-known definition of an elementary subgroup of a Chevalley group (or, more
generally, of a split reductive group), as well as several other definitions of an elementary subgroup
of isotropic classical groups and simple groups over fields~\cite{Bass,Tits64,VasO,VasSp,BakVav2000}.

By the structure constants of a root system we mean the structure constants of the corresponding
semisimple complex Lie group, or, in other words, constants appearing in the Chevalley commutator
formulas for the corresponding Chevalley group. They are among $\pm1$, $\pm 2$, $\pm 3$.

\begin{thm}\label{th:main}
Let $G$ be an isotropic reductive algebraic group over a commutative ring $R$.
Assume that for any maximal ideal $M$ of $R$ all irreducible components of the relative
root system of $G_{R_M}$ are of rank $\ge 2$, and,
if one of the irreducible components of the (usual) root system of $G_{R_M}$ is of type $B_2=C_2$ or $G_2$,
that the residue field $R_M/MR_M$ is not isomorphic to $\FF_2$.
Then $E(R)=[E(R),E(R)]$.
\end{thm}

Observe that the first condition of the theorem ensures that the
the elementary subgroup $E(R)$ of $G(R)$ is correctly defined, while the second one essentially eliminates
the well-known cases where the elementary subgroup of a {\it split} reductive group is not perfect.
Thus, the result is the strongest possible. One should note that, if we assume only that the rank of
relative root systems of $G_{R_M}$ is $\ge 1$, the question whether individual subgroups $E_P(R)$ are perfect
is of separate interest.

The proof of Theorem~\ref{th:main} is based on the notion of relative root subschemes (with respect to a parabolic subgroup)
of an isotropic group introduced in~\cite{PS},
the generalized Chevalley commutator formula~\cite[Lemma 9]{PS}, and localization in the
Quillen---Suslin style~\cite{Sus}. To shorten
the proof, we also make use of the classification of Tits indices of isotropic reductive groups
over local rings obtained in~\cite{PS-tind}.

\section{An abstract definition of relative roots}
In this section we recall the notion of an (abstract) system of relative roots introduced in~\cite{PS}
and prove a technical lemma.

Let $\Phi$ be a reduced root system in a Euclidean space with a scalar product $(-,-)$.
Let $\Pi=\{\alpha_1,\ldots,\alpha_l\}$ be a fixed system of simple roots of $\Phi$; if $\Phi$ is irreducible,
we assume that the numbering follows Bourbaki~\cite{Bu}. Let $D$ be the Dynkin diagram of $\Phi$.
We identify nodes of $D$ with the corresponding simple roots in $\Pi$. For a subgroup
$\Gamma\subseteq\Aut(D)$ and a $\Gamma$-invariant subset
$J\subseteq\Pi$, consider the projection
$$
\pi=\pi_{J,\Gamma}\colon\ZZ \Phi\longrightarrow \ZZ\Phi/\l<\Pi\setminus J;\ \alpha-\sigma(\alpha),\ \alpha\in J,\ \sigma\in\Gamma\r>.
$$
The set $\Phi_{J,\Gamma}=\pi(\Phi)\setminus\{0\}$ is called the system of {\it relative roots} corresponding to
the pair $(J,\Gamma)$.
The {\it rank} of $\Phi_{J,\Gamma}$ is the rank of $\pi(\ZZ \Phi)$ as a free abelian group.

It is clear that any relative root $A\in\Phi_{J,\Gamma}$ can be represented as a unique
linear combination of relative roots from $\pi(\Pi)$.
We say that $A\in\Phi_{J,\Gamma}$ is a {\it positive} (resp. {\it negative}) relative root, if it
is a non-negative (respectively, a non-positive) linear combination of the elements of $\pi(\Pi)$.
The sets of positive and negative relative roots will be denoted by
$\Phi_{J,\Gamma}^+$ and $\Phi_{J,\Gamma}^-$ respectively.
By the \emph{level} $\lev(A)$ of a relative root $A$ we mean the sum of coefficients in its decomposition.

Observe that $\Gamma$ acts on the set of irreducible components of the root system
$\Phi$. If this action is transitive, the system of relative roots
$\Phi_{J,\Gamma}$ is {\it irreducible}.
Clearly, any system of relative roots $\Phi_{J,\Gamma}$
is a disjoint union of irreducible ones; we call them the {\it irreducible components} of $\Phi_{J,\Gamma}$.

We will need the following lemma.
\begin{lem}\label{lem:decroots}
Let $\Phi$ be a root system with a fixed set of simple roots $\Pi$,
$\Gamma$ be a subgroup of $\Aut(D)$, and $J$ be a $\Gamma$-invariant subset of $\Pi$.
If a relative root $A\in\Phi_{J,\Gamma}$ lies in an irreducible component of rank $\ge 2$,
then there exist such non-collinear $B,C\in\Phi_{J,\Gamma}$ that $A=B+C$ and all relative roots
$iB+jC\in\Phi_{J,\Gamma}$, $i,j>0$, $(i,j)\neq (1,1)$, have the same sign as
$A$ and satisfy $|\lev (iB+jC)|>|\lev (A)|$.
\end{lem}
\begin{proof}
We can assume that the root system $\Phi$ is irreducible, and that $A$ is a positive relative root,
that is, $\pi^{-1}(A)\subseteq\Phi^+$.

Assume first that $A=k\pi(\alpha_r)$, where $\alpha_r\in\Pi$ is a simple root and $k>0$ is a positive
integer. Let $\alpha_s\in J$ be a simple root such that the $\Gamma$-orbits of $\alpha_s$ and $\alpha_r$
are distinct, and $\alpha_s$ is at the least possible distance from $\alpha_r$ on the Dynkin diagram.
It is easy to see that for any $\alpha\in\pi^{-1}(A)$ there exists $\beta\in\pi^{-1}(\alpha_s)$
such that $(\alpha,\beta)<0$, and, consequently, $\alpha+\beta\in\Phi$. Indeed, we have $m_s(\alpha)=0$
by definition, thus we can take for $\beta$ the sum of all simple roots
in
the Dynkin diagram chain between $\alpha_s$ and the nearest simple root appearing in the decomposition of $\alpha$.
Now set $B=\pi(\alpha+\beta)$ and $C=\pi(-\beta)$. It is clear that any root in
$\pi^{-1}(iB+jC)$, $i,j>0$, contains  the summand $i\alpha_r$ in its decomposition, and thus is
a positive root. Then $iB+jC$ is a positive relative root for any $i,j>0$. Moreover, one sees
that $\lev(iB+jC)=\lev(A)$ if and only if $i=j=1$. Since $\pi(\alpha)=k\pi(\alpha_r)$, and $\pi(-\beta)=-\pi(\alpha_s)$,
the roots $B$ and $C$ are non-collinear.

Consider the case where $A\neq k\pi(\alpha_r)$ for any $\alpha_r\in J$. For any $\alpha\in\pi^{-1}(A)$
there exists a sequence of simple roots $\beta_1,\ldots,\beta_n\in\Pi$ such that
$\alpha=\beta_1+\ldots+\beta_n$ and $\beta_1+\ldots+\beta_i\in\Phi$ for any $1\le i\le n$.
Let $i$ be the least possible index such that $\beta_{i+1},\ldots,\beta_n\in\Delta$. Then $\beta_i\in J$
and $\pi(\beta_1+\ldots+\beta_{i-1}+\beta_i)=A$. Set $B=\pi(\beta_1+\ldots+\beta_{i-1})$ and
$C=\pi(\beta_i)$. Since $B$ and $C$ are positive relative roots, we have $\lev(iB+jC)>\lev(A)$
for any $i,j>0$ distinct from $i=j=1$. The relative roots $B$ and $C$ are non-collinear since
otherwise we would have had $A=k\pi(\beta_i)$ for some $k>0$.
\end{proof}

\section{Isotropic reductive groups and relative root subschemes}
In this section we recall some basic notions pertaining to reductive groups over rings; see~\cite{SGA,PS,S-thes}
for more detailed exposition.

Let $R$ be a commutative ring with 1, and let $G$ be a reductive group scheme, or reductive group for short,
over $R$ (see~\cite{SGA} for the definition). We denote by
$G^{\ad}$ and $G^{\sc}$ the corresponding adjoint and simply connected semisimple groups, respectively.

Any reductive algebraic group $G$ over $R$ is split locally in the fpqc topology on $\Spec R$.
If $G$ is of constant type over $R$ (that is, the root system of $G$ is the same at any point of $\Spec R$),
then $G$ is a twisted form of a split reductive algebraic group $G_0$ over $R$,
given by a cocycle $\xi\in H^1_{fpqc}(R,\Aut(G_0))$. Recall that the connected component of $\Aut(G_0)$
is precisely $G_0^{\ad}$. The group $G$ is {\it of inner type}, if $\xi$ is in the image
of the natural map $H^1_{fpqc}(R,G_0^{\ad})\to H^1_{fpqc}(R,\Aut(G_0))$.
One can always find a finite Galois extension $S$ of $R$
such that $G_S$ is of inner type. The Galois group $\Gal(S/R)$ acts on the Dynkin digram of each
$G_{\overline{k(s)}}$, where $\overline{k(s)}$ is the algebraic closure of
the residue field of a point $s\in\Spec R$, via a {\it $*$-action} (see~\cite{PS-tind, S-thes}).

Recall that $G$ is called isotropic, if it contains a proper parabolic subgroup $P$ over $R$.
Recall that we call a parabolic subgroup $P$ of $G$ {\it strictly proper}, if
it does not contain any semisimple normal subgroup of $G$.
We set
$$
E_P(R)=\l<U_P(R),U_{P^-}(R)\r>,
$$
where $P^-$ is any parabolic subgroup of $G$ opposite to $P$, and $U_P$ and $U_{P^-}$ are the unipotent
radicals of $P$ and $P^-$ respectively.
The main theorem of~\cite{PS} states that $E_P(R)$
does not depend on the choice of a strictly proper parabolic subgroup $P$, as soon as
for any maximal ideal $M$ in $R$ all irreducible components of the relative
root system of $G_{R_M}$ are of rank $\ge 2$. Under this assumption, we call $E_P(R)$ the elementary
subgroup of $G(R)$ and denote it simply by $E(R)$.

Let $P=P^+$ be a parabolic subgroup of $G$, and $P^-$ be an opposite parabolic subgroup. Let $L=P^+\cap P^-$
be their common Levi subgroup.
It was shown in~\cite{PS} that we can represent $\Spec(R)$ as a finite disjoint union
$$
\Spec(R)=\coprod\limits_{i=1}^m\Spec(R_i),
$$ so that
the following conditions hold for $i=1,\ldots, m$:\\
\indent$\bullet$ for any $s\in\Spec R_i$ the root system of $G_{\overline{k(s)}}$ is the same;\\
\indent$\bullet$ for any $s\in\Spec R_i$ the type of the parabolic subgroup $P_{\overline{k(s)}}$ of
$G_{\overline{k(s)}}$ is the same;\\
\indent$\bullet$ if $S_i/R_i$ is a Galois extension of rings such that $G_{S_i}$ is of inner type,
then for any $s\in\Spec R_i$ the Galois group $\Gal(S_i/R_i)$ acts on the Dynkin diagram $D_i$ of
$G_{\overline{k(s)}}$ via the same subgroup of $\Aut(D_i)$.

From here until the end of this section, assume that $R=R_i$ for some $i$ (or just extend the base).
Denote by $\Phi$ the root system of $G$, by $\Pi$ a set of simple roots of $\Phi$, by $D$
the corresponding Dynkin diagram. Then the $*$-action on $D$ is determined by a subgroup $\Gamma$
of $\Aut D$. Let $J$ be the subset of $\Pi$ such that $\Pi\setminus J$ is the type
of $P_{\overline{k(s)}}$ (that is, the set of simple roots of the Levi sugroup $L_{\overline{k(s)}}$).
Then $J$ is $\Gamma$-invariant.
The system of relative roots $\Phi_{J,\Gamma}$ is called {\it the system of relative roots corresponding to $P$}
and denoted also by $\Phi_P$. If $R$ is a local ring and $P$ is a minimal parabolic subgroup of $G$,
then $\Phi$ can be identified with the relative root system of $G$ in the sense of~\cite[Exp. XXVI \S 7]{SGA}
(see also~\cite{BoTi} for the field case), as was shown in~\cite{PS,S-thes}.

To any relative root $A\in\Phi_P$
one associates a finitely generated projective $R$-module $V_A$ and a closed embedding
$$
X_A:W(V_A)\to G,
$$
where $W(V_A)$ is the affine group scheme over $R$ defined by $V_A$,
which is called a {\it relative root subscheme} of $G$.
These subschemes possess several nice properties
similar to that of elementary root subgroups of a split group, see~\cite[Th. 2]{PS}. In particular,
they  are subject to
certain commutator relations which generalize the Chevalley commutator formula.

More precisely, assume that $A,B\in\Phi_P$ satisfy $mA\neq -kB$ for any $m,k\ge 1$.
Then there exists a polynomial map
$$
N_{ABij}\colon V_A\times V_B\to V_{iA+jB},
$$
homogeneous of degree $i$ in the first variable and of degree $j$ in the second
variable, such that for any $R$-algebra $R'$ and for any
for any $u\in V_A\otimes_R R'$, $v\in V_B\otimes_R R'$ one has
\begin{equation}\label{eq:Chev}
[X_A(u),X_B(v)]=\prod_{i,j>0}X_{iA+jB}(N_{ABij}(u,v))
\end{equation}
(see~\cite[Lemma 9]{PS}).

In a strict analogy with the split case, for any $R$-algebra $R'$ we have
$$
E(R')=\l<X_A(V_A\otimes_R R'),\ A\in\Phi_{P}\r>
$$
(see~\cite[Lemma 6]{PS}).

We will aslo use the following statement which is a slight extension of~\cite[Lemma 10]{PS}.

\begin{lem}\label{lem:comm}
Consider $A,B\in\Phi_P$ satisfying $A+B\in\Phi_P$ and $mA\neq -kB$ for any $m,k\ge 1$.
Denote by $\Phi_0$ an irreducible component of $\Phi$ such that $A,B\in\pi(\Phi_0)$.

{\rm (1)} In each of the following cases

\quad {\rm (a)} structure constants of $\Phi_0$ are invertible in $R$ (for example,
if $\Phi_0$ is simply laced);

\quad {\rm (b)} $A\neq B$ and $A-B\not\in\Phi_P$;

\quad {\rm (c)} $\Phi_0$ is of type $B_l$, $C_l$, or $F_4$,
and $\pi^{-1}(A+B)$ consists of short roots;

\quad {\rm (d)} $\Phi_0$ is of type $B_l$, $C_l$, or $F_4$, and there exist long roots $\alpha\in\pi^{-1}(A)$, $\beta\in\pi^{-1}(B)$
such that $\alpha+\beta$ is a root;

the map $\ N_{AB11}: V_A\times V_B\to V_{A+B}\ $
is surjective.

{\rm (2)} If $A-B\in\Phi_P$ and $\Phi_0\neq G_2$, then
$$
\im N_{AB11}+\im N_{A-B,2B,1,1}+\sum_{v\in V_B}\im (N_{A-B,B,1,2}(-,v))=V_{A+B},
$$
where $\im N_{A-B,2B,1,1}=0$ if $2B\not\in\Phi_P$.
\end{lem}
\begin{proof}
(1) By~\cite[Lemma 4]{PS} any $\gamma\in \pi^{-1}(A+B)$
decomposes as $\gamma=\alpha+\beta$, $\alpha\in \pi^{-1}(A)$, $\beta\in \pi^{-1}(B)$.
Let $S_{\tau}$ be any member of an affine fpqc-covering $\coprod \Spec S_\tau\to \Spec R$ that splits $G$.
Set
$$
\Psi=\{iA+jB\ |\ i,j>0,\ (i,j)\neq (1,1),\ iA+jB\in\Phi_P\}.
$$
Then in the notation of~\cite[Th. 2]{PS}, over $S_\tau$ the commutator $[X_A(e_{\alpha}),X_B(e_{\beta})]$, computed modulo the subgroup
$\l<X_C(V_C),\ C\in\Psi\r>$,
is of the form $x_{\gamma}(\pm c)=X_{A+B}(\pm c e_{\gamma})$, where $c=\pm 1,\pm 2,\pm 3$ is the corresponding structure
constant. If (a) holds, then $c$ is invertible. If
(b), (c) or (d) holds, then $c$ necessarily equals $\pm 1$. Indeed, in the
only dubious case (d) one should
note that, due to the transitive action of the Weyl group of the Levi subgroup on the roots of the same shape
(see~\cite{ABS}),
{\it any} long root $\gamma\in\pi^{-1}(A+B)$
decomposes as a sum of long roots.
Hence $c$ is always invertible.
This implies that $\im (N_{AB11})_{\tau}=V_{A+B}\otimes S_{\tau}$. Since $\im N_{AB11}$ is a submodule of
$V_{A+B}$ defined over the base ring, we have $\im N_{AB11}=V_{A+B}$.

(2) See~\cite[Lemma 10]{PS}.
\end{proof}

\begin{lem}\label{lem:comm2}
Suppose that $\Phi_0$ is an irreducible component of $\Phi$ such that $\Phi_0\cong C_l$,
$l>2$, $P$ is a parabolic subgroup of type $\Pi\setminus J$, where $J=\{\alpha_i,\alpha_l\}$,
$2i=l$ ($\alpha_i$ is short, $\alpha_l$ is long), so that $\Phi_{0,P}\cong C_2$.
Denote $\pi(\alpha_i)=A_1$ and $\pi(\alpha_l)=A_2$.
Then
$$
\im (0,N_{A_1,A_1+A_2,1,1})+\sum_{v\in V_A}\im f_v=V_{A_1+A_2}\oplus V_{2A_1+A_2},
$$
where $f_v=(N_{A_1,A_2,1,1}(v,-),N_{A_1,A_2,2,1}(v,-))\colon V_{A_2}\to V_{A_1+A_2}\oplus V_{2A_1+A_2}$.
\end{lem}
\begin{proof}
Let $S_{\tau}$ be any member of an affine fpqc-covering $\coprod \Spec S_\tau\to \Spec R$ that splits $G$.
Suppose that $\gamma\in\pi^{-1}(2A_1+A_2)$ is a short root. We can find short roots
$\alpha\in\pi^{-1}(A_1)$, $\beta\in\pi^{-1}(A_1+A_2)$ such that $\gamma=\alpha+\beta$.
Therefore,
$$[X_{A_1}(e_\alpha),X_{A_1+A_2}(e_\beta)]=
x_\gamma(\pm 1)=X_{2A_1+A_2}(\pm e_\gamma).
$$
Hence, $e_\gamma\in\im (N_{A_1,A_1+A_2,1,1})_\tau$.
Now let $\gamma\in\pi^{-1}(2A_1+A_2)$ be a long root. Take
$\alpha\in\pi^{-1}(A_1)$, $\beta\in\pi^{-1}(A_2)$ such that $\beta\neq\alpha_l$ and
$\gamma=2\alpha+\beta$ (note that $\alpha$ is short, $\beta$ is long). 
Therefore
$$
[X_{A_1}(e_\alpha),X_{A_2}(e_\beta)]=
x_{\alpha+\beta}(\pm 1)x_{2\alpha+\beta}(\pm 1)=X_{A_1+A_2}(\pm e_{\alpha+\beta})X_{2A_1+A_2}(\pm e_\gamma).
$$
Finally, any $\gamma\in\pi^{-1}(A_1+A_2)$ is a short root, so there exist short roots
$\alpha\in\pi^{-1}(A_1)$, $\beta\in\pi^{-1}(A_2)$ such that
$\alpha+\beta=\gamma$, hence $[X_{A_1}(e_\alpha),X_{A_2}(e_\beta)]=X_{A_1+A_2}(\pm e_\gamma)$.
Combining these results and noting that the modules in question are defined over the base ring,
we are done.
\end{proof}

\section{Proof of Theorem~\ref{th:main}}

Let $R$ be a commutative ring with 1, $G$ be a reductive group over $R$, $P$ be a strictly proper
parabolic subgroup of $G$.
For any ideal $I\subseteq R$ we write
$$
E_P(I)=\l<U_P(I),U_{P^-}(I)\r>\le G(R).
$$
We denote by $R[Y,Z]$ a ring of polynomials in two variables $Z$ and $Y$ over $R$.

\begin{lem}\label{lem:incl}
Let $G$ be a reductive group scheme over a commutative ring $R$, and let $P$ and $P'$ be two strictly
proper parabolic subgroups of $G$ such that $P\le P'$ or $P'\le P$. Then
for any integer $m>0$ there exists an integer $k>0$ such that
$$
E_P(Z^kR[Z])\le E_{P'}(Z^mR[Z]).
$$
\end{lem}
\begin{proof}
Without loss of generality, we can assume that over $R$ we have two sets of relative root
subschemes $X_A(V_A)$, $A\in\Phi_P$, and $X_B(V_B)$, $B\in\Phi_{P'}$, corresponding to $P$ and $P'$
respectively.

Then, if $P\le P'$, by~\cite[Lemma~12]{PS} there exists an integer $k>0$
such that for any $A\in\Phi_P$ and any $v\in V_A$
one can find relative roots
$B_i\in\Phi_{P'}$, elements $v_i\in V_{B_i}$, and integers
$n_i>0$ {\rm ($1\le i\le m$)},
such that
$$
X_A(Y^k v)=\prod\limits_{i=1}^m X_{B_i}(Y^{n_i} v_{i}),
$$
and hence $X_A(Y^k v)\in E_{P'}(YR[Y])$. Substituting $R$ by $R[Z]$ and $Y$ by $Z^m$, we obtain
$$
E_P(Z^kR[Z])\le E_{P'}(Z^mR[Z]).
$$

If, conversely, $P'\le P$, we have $U_P\le U_{P'}$.
Let $\Psi^\pm\subseteq\Phi^+$ be two closed sets of roots corresponding to $U_{P^\pm}$.
Then we have $\pi(\Psi^\pm)\subseteq\Phi_{P'}^\pm$, where $\pi:\Phi\to\Phi_{P'}$ is the canonical projection.
By~\cite[Lemma 6]{PS} the map
$$
X_\Psi\colon\WW\Bigl(\bigoplus_{A\in\pi(\Psi^\pm)}V_A\Bigr)\to\UU_{P^\pm},\quad
(v_A)_A\mapsto\prod_A X_A(v_A),
$$
where the product is taken in any fixed order respecting the level of relative roots in $\Phi_{P'}$,
is an isomorphism of schemes over $R$. Therefore, $U_{P^\pm}(Z^mR[Z])\le U_{P'^\pm}(Z^mR[Z])$.
\end{proof}

\begin{lem}\label{lem:local}
In the setting of Theorem~\ref{th:main}, assume moreover that $R$ is a local ring.
Then for any integer $m>0$ there exists an
integer $k>0$ such that for any $R$-algebra $R'$ one has
$$
E_P(Z^kR'[Z])\subseteq [E_P(Z^mR'[Z]),E_P(Z^mR'[Z])].
$$
\end{lem}
\begin{proof}
Let $\der(G)$ be the algebraic derived subgroup of the reductive group scheme $G$ (see~\cite[Exp. XXII, 6.2]{SGA}).
Then, clearly, $E_P(R)\subseteq \der(G)(R)$. Since $\der(G)$ is a semisimple group, we can assume that
$G$ is semisimple. Moreover, since the canonical projection $G^{sc}\to G$, where $G^{sc}$ is the simply
connected semisimple group corresponding to $G$, is surjective on $U^\pm(R)$, we can assume that $G$ is simply
connected. Any simply connected semisimple group
is a direct product of simply connected semisimple groups that cannot be decomposed into a product
of smaller semisimple groups.
These groups $G_i$, $i=1,\ldots,n$, are Weil restrictions
of certain simple reductive groups $G_i'$ over a finite \'etale extension $S$ of $R$: $G_i={\mathrm R}_{S/R}(G_i')$~\cite[Exp. XIV Prop. 5.10]{SGA}.
Note that the group of $R$-points $G_i(R)$ is canonically isomorphic to the group of $S$-points $G_i'(S)$.
This isomorphism also respects the embedding $P_i(R)\to G_i(R)$, for any parabolic subgroup $P_i$ of $G_i$.
Then, clearly, we can assume from the very beginning that $G$ is a simple reductive group, and the root
system $\Phi$ of $G$ is irreducible.

Note that by Lemma~\ref{lem:incl} we can substitute $P$ by any other strictly proper parabolic subgroup $P'$ of $G$
such that $P\le P'$ or $P'\le P$.
Further,  over $R$ we have a set of relative root
subschemes $X_A(V_A)$, $A\in\Phi_P$, corresponding to $P$.

We are going to show by induction on
$|\lev A|$ that for
any $A\in\Phi_P$ there exists an integer $k=k(A)>0$ such that for any $R$-algebra $R'$ and any
$v\in V_A\otimes_R R'$ one has
\begin{equation}\label{eq:Z^k}
X_A(Z^kv)\in [E_P(ZR'[Z]),E_P(ZR'[Z])].
\end{equation}
The claim of the lemma then follows by substituting $Z$ by $Z^m$ and $R'$ by $R[Z]$, and taking the final $k$
to be the maximum of all $k(A)$, $A\in\Phi_P$.

Recall that by Lemma~\ref{lem:decroots}
there exist non-collinear relative roots $B,C\in\Phi_P$ such that $A=B+C$ and
and all roots $iB+jC\in\Phi_{J,\Gamma}$, $i,j>0$, $(i,j)\neq (1,1)$, have the same sign as
$A$ and satisfy $|\lev (iB+jC)|>|\lev (A)|$. Assume
that the map $N_{BC11}:V_B\times V_C\to V_A$ is surjective. Then by the
generalized Chevalley commutator formula~\eqref{eq:Chev} we have
that for any $R$-algebra $R'$, and any $v\in V_A\otimes_R R'[Z]$,
$$
X_A(Z^k v)=[X_B(Zu_{10}),X_C(Z^{k-1}u_{01})]\cdot\prod\limits_{
\begin{array}{c}
\st i,j>0;\\
\st (i,j)\neq(1,1)
\end{array}
}
X_{iB+jC}(Z^{i+j(k-1)}u_{ij})
$$
for some $u_{ij}\in V_{iB+jC}\otimes_R R'[Z]$, $i,j>0$, $(i,j)\neq (1,1)$. Then, by the inductive hypothesis,
\eqref{eq:Z^k} holds for $k$ large enough.

Now for any relative root $A\in\Phi_P$ (for a suitable choice of the  parabolic subgroup $P$) we either
show that for {\it any} decomposition $A=B+C$, where $B$ and $C$ are non-collinear, the map $N_{BC11}$ is surjective;
or provide an explicit decomposition of $X_A(Z^k v)$, $v\in V_A\otimes_R R'$, into a product
of commutators in $E_P(ZR'[Z])$, so that~\eqref{eq:Z^k} is satisfied for $k$ large enough.

Assume first that all structure constants of the root system $\Phi$ of $G$ are invertible in $R$;
this includes the case where $\Phi$ is simply laced. Then
by Lemma~\ref{lem:comm} the map $N_{BC11}$ is surjective
for any decomposition $A=B+C$, where $B$ and $C$ are non-collinear.

Consider the case where $\Phi=\Phi_P=C_2$, so $G$ is split. Let $\Pi=\{A_1,A_2\}$, $\Phi^+=\{A_1,A_2,A_1+A_2,2A_1+A_2\}$.
Let $M$ be the maximal ideal of $R$.
By the hypothesis of Theorem~\ref{th:main}, $R/M\not\cong F_2$,
hence we can take $\eps\in R\setminus M$ such that $\eps^2-\eps\in R\setminus M=R^*$.
If the root $A\in\Phi_P$ is long, we can assume that $A=2A_1+A_2$. Let
$$
g_1(s,t)=[X_{A_1}(s),X_{A_2}(t)]=X_{A_1+A_2}(st)X_{2A_1+A_2}(s^2t)
$$
and
$$
g_2(s,t,u)=[X_{A_2}(u),[X_{A_1+A_2}(s),X_{-A_2}(t)]]=X_{A_1+A_2}(-stu)X_{2A_1+A_2}(-s^2t^2u).
$$
Therefore,
$$
g_1(Z^2,-Z^{k-4}\eps(\eps^2-\eps)^{-1}v)
\cdot g_2(Z,Z\eps,-Z^{k-4}(\eps^2-\eps)^{-1}v)=X_{2A_1+A_2}(Z^kv).
$$
If the root $A\in\Phi_P$ is short, we can assume that $A=A_1+A_2$, hence
$$
g_1(Z,Z^{k-1}v)\cdot X_{2A_1+A_2}(-Z^{k+1}v)=X_{A_1+A_2}(Z^kv).
$$
This means that~\eqref{eq:Z^k} holds for these roots for any $k\ge 5$.

Consider the case where $\Phi=\Phi_P=G_2$, so $G$ is split. Let $\Pi=\{A_1,A_2\}$,
$\Phi^+=\{A_1,A_2,A_1+A_2,2A_1+A_2,3A_1+A_2,3A_1+2A_2\}$.
By the hypothesis of Theorem~\ref{th:main}, $R/M\not\cong F_2$,
hence we can take $\eps\in R\setminus M$ such that $\eps^2-\eps\in R\setminus M=R^*$.
If the root $A$ is long, we can assume that $A=3A_1+2A_2$. Then
$$
[X_{A_2}(Zv),X_{3A_1+A_2}(Z^{k-1})]=X_{3A_1+2A_2}(Z^kv).
$$
Therefore,~\eqref{eq:Z^k} holds for long roots for any $k\ge 2$.
If the root $A$ is short, we can assume that $A=2A_1+A_2$. Then
$$
[X_{A_1}(s),X_{A_2}(t)]=X_{A_1+A_2}(st)\cdot X_{2A_1+A_2}(s^2t)\cdot X_{3A_1+A_2}(s^3t)\cdot X_{3A_1+2A_2}(s^3t^2).
$$
Hence,
\begin{multline}
[X_{A_1}(Z\eps),X_{A_2}(-(\eps^2-\eps)^{-1}Z^{k-2}v)]^{-1}\cdot
[X_{A_1}(Z),X_{A_2}(-\eps(\eps^2-\eps)^{-1}Z^{k-2}v)]=\\
X_{2A_1+A_2}(Z^kv)X_{3A_1+A_2}((\eps+1)Z^{k+1}v)X_{3A_1+2A_2}(\eps(\eps^2-\eps)^{-1}Z^{2k+1}v),
\end{multline}
and the roots $3A_1+A_2$ and $3A_1+2A_2$ are long.
This means that~\eqref{eq:Z^k} holds for short roots for any $k\ge 3$.

We are left with the case where $\Phi$ is of type $B_l$, $C_l$, or $F_4$.
Recall that by the hypothesis of Theorem~\ref{th:main} $G$ contains a split torus of rank $\ge 2$.
Hence in the $F_4$ case the classification of Tits indices over local rings~\cite{PS-tind} says that
$G$ is a split group. Hence we can assume that $P$ is a Borel subgroup of $G$,
and $\Phi_P=\Phi$ is a root system of type $F_4$.
Then if the root $A\in\Phi_P$ is short, by Lemma~\ref{lem:comm} the map $N_{BC11}$ is surjective
for any non-collinear $B,C\in\Phi_P$ such that $A=B+C$. If this root is long, then it belongs to the long root subsystem of
$F_4$, which has type $D_4$. Then (for example, by Lemma~\ref{lem:decroots}) we can find two long roots
$B,C\in\Phi_P$, such that $B+C=A$, and necessarily, $iB+jC$ is not a root
for any $i,j>0$ distict from $i=j=1$. Then
$$
X_A(Z^k v)=[X_B(Zu_{10}),X_C(Z^{k-1}u_{01})]
$$
for some $u_{10}\in V_{B}\otimes_R R'[Z]$ and $u_{01}\in V_C\otimes_R R'[Z]$.

Consider the case where $\Phi$ is of type $B_l$, $l\ge 3$.
By the classification of Tits indices over local rings~\cite{PS-tind}, we can assume that $P$ is a
parabolic subgroup of type $\Pi\setminus J$, where $\Pi$ is a set of simple roots of $\Phi$ and
$J=\{\alpha_1,\alpha_2\}$. Then $\Phi_P$ can be identified with a root system of type $B_2$. One readily
sees, using the fact that $l\ge 3$,
that for any relative root $A\in\Phi_P$ and any pair $B,C\in\Phi_P$ satisfying $A=B+C$, we can find
a pair of long roots $\beta\in\pi^{-1}(B)$, $\gamma\in\pi^{-1}(C)$ such that $\beta+\gamma$ is a root
(one can assume that $A$ is one of two simple roots of $\Phi_P$, due to the lifting of the relative
Weyl group~\cite[Exp. XXVI Th. 7.13 (ii)]{SGA}). Then by Lemma~\ref{lem:comm}
the map $N_{BC11}$ is surjective.

It remains to consider the case where $\Phi$ is of type $C_l$, $l\ge 3$.
First assume that $P$ is a parabolic subgroup of type $\Pi\setminus J$, where $J=\{\alpha_i,\alpha_j\}$
for two short simple roots $\alpha_i,\alpha_j$ of $\Phi$.
Then $\Phi_P$ can be identified with a root system of type $BC_2$. One readily sees that for all
extra-short and short relative roots $A\in\Phi_P$ the set $\pi^{-1}(A)$ consists of short roots, and
hence by Lemma~\ref{lem:comm} the map $N_{BC11}$ is surjective for any decomposition $A=B+C$.
Let $A$ be a long root. Let $A_1$ and $A_2$ be a short and an extra-short simple roots of $\Phi_P$.
We can assume
without loss of generality that $A=2A_1+2A_2$. Take $k\ge 4$. Then by Lemma~\ref{lem:comm} (2) and by the generalized
Chevalley commutator formula, for any
$R$-algebra $R'$, and any $v\in V_A\otimes_R R'[Z]$, we have
$$
X_A(Z^k v)=[X_{A_1}(Zu_1),X_{2A_2}(Z^{k-2}u_2)]\cdot X_{A_1+2A_2}(Z^{k-1}u_3)
$$
for some $u_1\in V_{A_1}\otimes_R R'[Z]$, $u_2\in V_{2A_2}\otimes_R R'[Z]$, and
$u_3\in V_{A_1+2A_2}\otimes_R R'[Z]$. Further, by Lemma~\ref{lem:comm} (1) and by the generalized
Chevalley commutator formula, since $\pi^{-1}(A_1+2A_2)$ consists of short roots, we have
$$
X_{A_1+2A_2}(Z^{k-1}u_3)=[X_{A_1+A_2}(Zu_4),X_{A_2}(Z^{k-3}u_5)]
$$
for some $u_4\in V_{A_1+A_2}\otimes_R R'[Z]$ and $u_5\in V_{A_2}\otimes_R R'[Z]$.
Hence~\eqref{eq:Z^k} holds for $A$ for any $k\ge 4$.

By the classification of Tits indices over local rings~\cite{PS-tind}, the only remaining case
is when $P$ is a parabolic subgroup of type $\Pi\setminus J$ for $J=\{\alpha_i,\alpha_l\}$,
where $l=2i$. Now $\alpha_i$ is short, $\alpha_l$ is long, and $\Phi_P$ can be identified with
a root system of type $B_2=C_2$. As in Lemma~\ref{lem:comm2}, we put $A_1=\pi(\alpha_i)$,
$A_2=\pi(\alpha_l)$.
Then if the root $A\in\Phi_P$ is short, by the lifting of the relative Weyl group~\cite[Exp. XXVI Th. 7.13 (ii)]{SGA}),
we can assume that $A=A_1+A_2$. By Lemma~\ref{lem:comm2} for any $R$-algebra $R'$, and any $v\in V_{A}\otimes_R R'[Z]$,
we have
$$
X_{A}(Z^k v)=\prod_i[X_{A_1}(Zv_i),X_{A_2}(Z^{k-1}u_1)]
$$
for some $u_1\in V_{A_2}\otimes_R R'[Z]$, $v_i\in V_{A_1}\otimes_R R'[Z]$.
If the root $A\in\Phi_P$ is long, we can assume that $C=2A_1+A_2$. By Lemma~\ref{lem:comm2} for any $R$-algebra $R'$, and any
$v\in V_{2A_1+A_2}\otimes_R R'[Z]$, we have
$$
X_{A}(Z^k v)=[X_{A_1}(Zu_1),X_{A_1+A_2}(Z^{k-1}u_2)]\cdot\prod_i[X_{A_1}(Zv_i),X_{A_2}(Z^{k-2}u_3)]
$$
for some $u_1\in V_{A_1}\otimes_R R'[Z]$, $u_2\in V_{A_1+A_2}\otimes_R R'[Z]$,
$u_3\in V_{A_2}\otimes_R R'[Z]$, $v_i\in V_{A_1}\otimes_R R'[Z]$.
Hence~\eqref{eq:Z^k} holds for $A$ for any $k\ge 3$.
\end{proof}

\begin{proof}[Proof of Theorem~\ref{th:main}]
Recall that we can represent $\Spec(R)$ as a finite disjoint union $\Spec(R)=\coprod\limits_{i=1}^n\Spec(R_i)$, so that
$E(R)=\prod\limits_{i=1}^n E(R_i)$ and for any $1\le i\le n$ we have
$$
E(R_i)=\l<X_A(V_A),\ A\in\Phi_{P_{R_i}}\r>,
$$
for a set of root subschemes $X_A$, $A\in\Phi_{P_{R_i}}$, over $R_i$. Hence we can assume that
$$
E(R)=\l<X_A(V_A),\ A\in\Phi_{P}\r>
$$
from the very beginning.

We show that $[E(R),E(R)]$ contains any $X_A(v)$, $A\in \Phi_P$, $v\in V_A$, by induction
on $|\lev A|$.
Take
$$
I=\{s\in R\ |\ X_A(tsv)\in [E(R),E(R)]\ \forall\ t\in R\}.
$$
By~\cite[Th. 2]{PS} for any $u,u'\in V_A$ we have
$$
X_A(u+u')=X_A(u)X_A(u')\prod_{i>0}X_{iA}(u_i)
$$
for some $u_i\in V_{iA}$. Hence by the inductive hypothesis $I$ is an ideal
of $R$.
If $I\neq R$, let $M$ be a maximal ideal of $R$ containing $I$. Let $F_M$ denote
both the localization homomorphism $R\to R_M$ and the induced homomorphism $G(R[Y,Z])\to G(R_M[Y,Z])$.
By Lemma~\ref{lem:local}
there is an $m>0$ such that
we can represent the element $F_M(X_A(Z^mYv))$ of $G(R_M[Y,Z])$ as a product
$$
F_M(X_A(Z^mYv))=\prod_{i=1}^n [h_i(Y,Z),g_i(Y,Z)],
$$
for some $h_i(Y,Z),g_i(Y,Z)\in E(ZR_M[Y,Z])$, $1\le i\le n$.
By~\cite[Lemma 15]{PS} there exist
$$
h'_i(Y,Z),\,g'_i(Y,Z)\in E_P(R[Y,Z],ZR[Y,Z])\le E_P(R[Y,Z])=E(R[Y,Z])
$$
and $s\in R\setminus M$ such that $F_M(h'_i(Y,Z))=h_i(Y,sZ)$ and $F_M(g'_i(Y,Z))=g_i(Y,sZ)$ for all $i$.
Hence
$$
F_M\bigl(X_A((sZ)^mYv)\bigr)=F_M\Bigl(\prod_{i=1}^n [h'_i(Y,Z),g'_i(Y,Z)]\Bigr).
$$
Then
by~\cite[Lemma 14]{PS} there exists $t\in R\setminus M$ such that
$$
X_A((tsZ)^mYv)=\prod_{i=1}^n [h'_i(Y,tZ),g'_i(Y,tZ)].
$$ Substituting $Z$ by $1$ and $Y$ by an arbitrary element
of $R$, we see that $(ts)^m\in I$. But $(ts)^m\in R\setminus M$, a contradiction.
\end{proof}

The authors are sincerely grateful to Victor Petrov and Nikolai Vavilov for their inspiring comments on
the subject of this paper.

\renewcommand{\refname}{References}


\begin{thebibliography}{99}


\bibitem{ABS} H.~Azad, M.~Barry, G.~Seitz, {\it On the structure of parabolic subgroups},
Comm. in Algebra {\bf 18} (1990), 551--562.


\bibitem{Bass} H.~Bass, {\it K-theory and stable algebra}, Publ. Math. I.H.\'E.S.
{\bf 22} (1964), 5--60.


\bibitem{BakVav2000} A.~Bak, N.~Vavilov, {\it Structure of
hyperbolic unitary groups {\rm I}. Elementary subgroups},
Algebra Colloquium {\bf 7} (2000), 159--196.


\bibitem{BoTi} A.~Borel, J.~Tits, {\it Groupes r\'eductifs}, Publ. Math. I.H.\'E.S.
{\bf 27} (1965), 55--151.

\bibitem{Bu} N.~Bourbaki, {\it Groupes et alg\`ebres de Lie. Chapitres 4--6}, Hermann, Paris, 1968.

\bibitem{SGA} M.~Demazure, A.~Grothendieck, {\it Sch\'emas en groupes}, Lecture Notes in
Mathematics, Vol. 151--153, Springer-Verlag, Berlin-Heidelberg-New York, 1970.

\bibitem{PS} V. Petrov, A. Stavrova, {\it Elementary subgroups of isotropic reductive groups},
St. Petersburg Math. J. {\bf 20} (2009), 625--644.

\bibitem{PS-tind} V.~Petrov, A.~Stavrova, Tits indices over semilocal rings,
Preprint \hbox{POMI}, N 2 (2009), 1--22.

\bibitem{S-thes} A.~Stavrova, Stroenije isotropnyh reduktivnyh grupp, PhD thesis, St. Petersburg
State University, 2009.

\bibitem{Stein} M.R. Stein, Generators, relations, and coverings of Chevalley groups over commutative rings,
Amer. J. Math. {\bf 93} (1971), 965--1004.

\bibitem{Sus} A.A.~Suslin,
{\it On the structure of the special linear group over polynomial rings}, Math. USSR Izv. {\bf 11}
(1977), 221--238.


\bibitem{Tits64} J.~Tits, {\it Algebraic and abstract simple groups}, Ann. of Math. {\bf 80}
(1964), 313--329.

\bibitem{VasO} L.N.~Vaserstein, \emph{Normal subgroups of orthogonal groups over commutative
rings}, Amer. J. Math. \textbf{110} (1988), 955--973.


\bibitem{VasSp} L.N.~Vaserstein, \emph{Normal subgroups of symplectic groups over rings}, $K$-Theory
\textbf{2} (1989), 647--673.





\end{thebibliography}
\end{document}